\documentclass[11pt]{amsart}

\usepackage{amsmath}
\usepackage{amssymb}
\usepackage{graphicx}

\newtheorem{theorem}{Theorem}[section]
\newtheorem{corollary}[theorem]{Corollary}
\newtheorem{lemma}[theorem]{Lemma}
\newtheorem{proposition}[theorem]{Proposition}
\theoremstyle{definition}
\newtheorem{definition}[theorem]{Definition}

\newtheorem{remark}[theorem]{Remark}

\newcommand{\R}{\mathbb{R}}
\newcommand{\N}{\mathbb{N}}

\newcommand{\bd}{\mathrm{bdry}\,}
\newcommand{\epi}{\mathrm{epi} }

\newcommand{\gh}{\mathrm{graph}\,}

\numberwithin{equation}{section}

\title[Strict Hadamard  differentiability of sets]{Geometric characterizations of the strict Hadamard differentiability of sets}

\author[A. Jourani]{Abderrahim Jourani}

\address[A. Jourani]{ \\ \small Universit\'e de  Bourgogne Franche-Comt\'e\\ \small Institut de Math\'ematiques de Bourgogne \\ \small UMR 5584, CNRS, , 21000 Dijon, France\\}
\email{{\tt \small abderrahim.jourani@u-bourgogne.fr}}

\author[M. Sene]{Moustapha Sene}
\address[M. Sene]{\small Moustapha S\'ene\\
\small Universit\'e Gaston Berger\\
\small D\'epartement de Math\'ematiques\\
\small Saint-Louis, S\'en\'egal\\}
\email{\tt moustapha2.sene@ugb.edu.sn}



\begin{document}

\begin{abstract}
Let $S$ be a closed subset of a Banach space $X$. Assuming that $S$ is epi-Lipschitzian at $\bar{x}$ in the   boundary $ \bd S$ of $S$, 
 we  show  that $S$ is strictly Hadamard  differentiable at $\bar{x}$ IFF the Clarke tangent cone $T(S, \bar{x})$ to $S$ at $\bar{x}$ contains a closed  hyperplane IFF the Clarke  tangent cone $T(\bd S, \bar{x})$ to $\bd S$ at $\bar{x}$  is a closed hyperplane. Moreover when $X$ is of finite dimension, $Y$ is a  Banach space and  $g: X \mapsto Y$ is a locally Lipschitz mapping around $\bar{x}$,
  we show that $g$ is strictly Hadamard  differentiable at $\bar{x}$  IFF $T(\mathrm{graph}\,g, (\bar{x}, g(\bar{x})))$ is isomorphic to $X$ IFF the set-valued mapping $x\rightrightarrows K(\gh g, (x, g(x)))$ is continuous at $\bar{x}$  and $K(\gh g, (\bar{x}, g(\bar{x})))$ is isomorphic to $X$, where $K(A, a)$ denotes the contingent cone to a set $A$ at $a \in A$. 

\end{abstract}

\maketitle

\section{Introduction}
\label{intro}
Let $X$ be a real normed vector space and $S$ be a closed subset with $\bar{x}\in S$. A vector $v\in X$ is said to be a {\it Clarke tangent vector} to $S$ at $\bar{x}$ if for any neighborhood $V$ of $v$ there exist a neighborhood $U$ of $\bar{x}$ and some $\lambda>0$ such that for all 
$(t,x)\in ]0,\lambda[\times (U\cap C)$ we have 
$$(x+tV)\cap S\neq\emptyset.$$
The characterization in terms of sequences can be stated as follows :  $v\in X$ is a Clarke tangent vector to $S$ at $\bar{x}$ if and only if for any sequence 
$(x_n)_{n\in\N}$ of $S$ converging to $\bar{x}$ and any sequence of positive reals $(t_n)_{n\in\N}$  converging to $0$, there exists a sequence $(v_n)_{n\in\N}$ in $X$ converging to $v$ such that $ x_n + t_nv_n\in S$ for all $n\in \N$ sufficiently large. The collection  $T(S,\bar{x})$ of all such vectors is a closed convex cone which is called the Clarke tangent cone to $S$ at $\bar{x}$ (see \cite{7}). When $\bar{x}\notin S$ one puts,  $T(S,\bar{x})=\emptyset$.\\
The Clarke tangent cone is involved in many geometrical representation of sets such as locally compact and definable sets in $\R^n$. For instance in \cite{2} the authors proved that a nonempty locally compact subset $S\subset\R^n$ is a $C^1$-manifold if and only if the Clarke tangent cone and the upper paratangent cone to $S$ coincide at every point. More historical characterizations of $C^1$-manifolds in Euclidean spaces by tangent cones are reviewed in \cite{1}. 
More recently, the authors in \cite{20} have established a geometrical characterization of definable sets in terms of tangent cones.\\
It is our purpose in this paper to provide  geometrical characterizations of the strict Hadamard  differentiability of epi-Lipschitz and non epi-Lipschitz sets  by tangent cones. Indeed we show that an epi-Lipschitz set $S$ at $\bar{x}$ is strictly Hadamard      differentiable at $\bar{x}$ if and only if the Clarke tangent cone to $S$ at $\bar{x}$ contains a closed hyperplane if and only if  the Clarke  tangent cone $T(\bd S, \bar{x})$ to the boundary of $S$ at $\bar{x}$ is a closed hyperplane. 
Moreover in finite dimension,  we establish that a mapping $g:X\rightarrow Y$ is strictly Hadamard  differentiable at $\bar{x}$ if and only if $T(\mathrm{graph}\,g, (\bar{x}, g(\bar{x})))$ is isomorphic to $X$. Further corollaries are provided throughout the paper.\\
The paper is organized as follows. Section 2 is devoted to some nonsmooth analysis tools. In section 3, we present epi-Lipschitz sets and define the concepts of strictly Hadamard  differentiability of sets. Then, we study the inverse image of an epi-Lipschitz set by a strictly Hadamard  differentiable mapping. Section 4 and 5 are respectively devoted to the strict Hadamard  differentiability of epi-Lipschitz and non epi-Lipschitz sets.
\section{Tools of nonsmooth analysis}

Let $S$ be a subset of a vector normed space $X$. The Clarke normal cone to $S$ at $\bar{x}\in S$ is the negative polar of $T(S, \bar{x})$, that is 
$$N(S, \bar{x}) = \{ x^*\in X^* : \langle x^*, h\rangle \leq 0,\, \forall\, h\in T(S, \bar{x})\}$$
where $X^*$ denotes the topological dual space of $X$ and $\langle \cdot, \cdot\rangle$ is the pairing between $X$ and $X^*$. Having defined this object, we may now introduce the Clarke subdifferential. Let $f : X \mapsto \R\cup\{ +\infty\}$ be a lower semicontinuous function and $\bar{x}\in X$, with $f(\bar{x})<\infty$. The Clarke subdifferential of $f$ at $\bar{x}$ is given  (\cite{8}) by 
$$\partial f(\bar{x}) = \{ x^*\in X^* : \quad (x^*, -1) \in N(\hbox{epi} f, (\bar{x}, f(\bar{x})))\}$$
where $\hbox{epi} f$ denotes the epigraph of $f$, that is,
$$\hbox{epi} f = \{ (x, r)\in X\times \R : \quad f(x) \leq r\}.$$
It is natural to put $\partial f(\bar{x})  = \emptyset$ whenever $f(\bar{x})=\infty$. When $f$ is locally Lipschitz around $\bar{x}$, the Clarke subdifferential has the following equivalent definition
$$\partial f(\bar{x}) = \{ x^*\in X^* : \quad \langle x^*, h\rangle \leq \limsup_{{t\to 0^+}\atop{x\to \bar{x}}} {{f(x+th)-f(x)}\over t}\quad \forall\, h\in X\}.$$
It is known (\cite{8}), that whenever $f$ is locally Lipschitz around $\bar{x}$, the following equivalence holds true 
$$\partial f(\bar{x}) \hbox{ is a singleton IFF }  f  \hbox{ is strictly Hadamard differentiable at $\bar{x}$}.$$
We recall  that a mapping $g : X \mapsto Y$, with $Y$ a vector normed space, is strictly Hadamard differentiable at $\bar{x}$ if there exists a linear continuous mapping $Dg(\bar{x}) : X\to Y$ such that 
$$\lim_{{{t \to 0^+}\atop{x\to \bar{x}}}\atop{u\to h}} {{g(x+tu) - g(x)}\over t} = Dg(\bar{x})(h)$$
The mapping $Dg(\bar{x})$ is called  the strict Hadamard  derivative of $g$ at $\bar{x}$. 

In order to apply  Graves Theorem, we need an other concept of differentiability. We recall (\cite{22}) that a mapping $g : X \mapsto Y$, with $Y$ a vector normed space, is strictly Fr\'echet differentiable at $\bar{x}$ if 
$$\lim_{{x\to \bar{x}}\atop{x'\to \bar{x}}} {{g(x) - g(x') - Dg(\bar{x})(x-x')}\over {\Vert x-x'\Vert}} = 0$$
where $Dg(\bar{x})$ denotes the strict Fr\'echet derivative of $g$ at $\bar{x}$. \\
It is not difficult to see that each strictly Fr\'echet  differentiable mapping is strictly  Hadamard  differentiable and both concepts coincide in finite dimensional spaces for locally Lipschitz mappings. This equivalence is no longer true in the  infinite dimensional situation. Borwein \cite{5} have established an equivalence of both concepts for locally Lipschitz functions with an additional condition. 

\section{Epi-Lipschitz sets}
Let $S$ be a subset of a normed space $X$. Following Rockafellar (see. \cite{22,23}), $S$ is said to be epi-Lipschitzian at $\bar{x} \in S$ in a direction $\bar{v}\neq0$  if there exists   $\gamma>0$  such that the following inclusion  holds 
$$S\cap B(\bar{x}, \gamma) + t B(\bar{v}, \gamma)\subset S\quad \forall\, t\in ]0, \gamma[.$$
If $S$ is epi-Lipschitzian at any of its points then it is called epi-Lipschitzian.\\
Assume that $S$ is epi-Lipschitzian at $\bar{x}$ in a direction $\bar{v}\neq0$. In the case of finite dimensional real normed spaces, Rockafellar (cf. \cite[$\S$ 4]{24}) proved the existence of a closed hyperplane $E$ such that
$X=E\oplus\R\bar{v}$, a neighbourhood $\Omega$ of $\bar{x}$ in $X$ and a function $f:E\rightarrow \R$ locally Lipschitzian near $\pi_E(\bar{x})$ (where $\bar{x}=\pi_E(\bar{x})+\pi_{\bar{v}}(\bar{x})$ with 
$\pi_E(\bar{x})\in E$ and $\pi_{\bar{v}}(\bar{x})\in \R\bar{v}$ uniquely defined) such that
$$\Omega\cap S=\Omega\cap\{u+r\bar{v}:u\in E,r\in\R, f(u)\leq r\}.$$
Further he observed in \cite[$\S$ 5]{25} that if $S$ is closed then the arguments in \cite{24} are also true in the case of infinite dimensional real normed space. 
Having this, one can write
\begin{equation}\label{epiphi}
\Omega\cap S=\Omega\cap \varphi(\epi\,f)\quad\mbox{where}\quad \varphi(u,r):=u+r\bar{v}
\end{equation}
This means that $S$ can locally be represented as the epigraph of locally Lipschitz function $f$.
\begin{eqnarray}\label{local}
\hbox{\it  We call $f$ a local representation of $S$ near  $\bar{x}$.}
\end{eqnarray}
In \cite{24}, Rockafellar showed that 
$$\hbox{a closed set $S\subset X$ is epi-Lipschitzian at $\bar{x}$ IFF int$T(S, \bar{x}) \neq \emptyset$}$$
 provided that $X$ is of finite dimension. Borwein and Strojwas \cite{3} established that this characterization remains valid in infinite dimensional Banach spaces provided that $S$ is compactly epi-Lipschitzian at $\bar{x}$. Following Borwein and Strojwas  \cite{3}, a set $S \subset X$ is said  to be 
compactly epi-Lipschitzian (CEL) at $\bar{x}$ if there exist $\gamma>0$ and a norm-compact set $H$ such that
\begin{eqnarray}\label{CEL}
S\cap B(\bar{x}, \gamma) + tB(0, 1) \subset S + tH, \, \forall\; t\in [0, \gamma].
\end{eqnarray}
It is clear that this class of sets  encompasses that of epi-Lipschitz sets. Note that in finite dimensions, every nonempty set is CEL. \\
Further representations of closed epi-Lipschitz sets in terms of sub-level of Lipschitz functions are established in Czarnecki and Thibault \cite{10}. We recall some of them in the following.\\
Let $X$ be a real normed space and $S\subset X$. The {\it signed distance function} to $S$ is defined by
$$\Delta_S(x):=d_S(x)-d_{S^c}(x)\quad\mbox{for all}\quad x\in X,$$
where $S^c=X\setminus S$ and 
$$d_S(x) := \inf_{u\in S}\Vert u-x\Vert$$
is the distance function to the set  $S$.  It is easily seen that 
$$\mathrm{cl}S=\{x\in X: \Delta_S(x)\leq0\}\quad\mbox{and}\quad\bd S=\{x\in X:\Delta_S(x)=0\}.$$
Moreover it is known (see e.g.,\cite{13}) that $\Delta_S$ is $1$-Lipschitzian on $X$. Czarnecki and Thibault \cite{10} established the following local sublevel representation.
\begin{theorem}\label{tcz1}\cite{10}
Let $S$ be a subset of a real normed space $X$ and $\bar{x}\in S\cap\bd S$. Assume with $\overline{S}:=\mathrm{cl}S$ that $\mathrm{int}\overline{S}\cap U\subset S$ for some neighbourhood $U$ of $\bar{x}$. Then, the set $S$ is epi-Lipschitzian at $\bar{x}$ if and only if $0\notin\partial \Delta_S(\bar{x})$.
\end{theorem}
This allows to state the following known result.
\begin{corollary}\label{cor1}
Let $S$ be a subset of a normed space $X$ which is epi-Lipschitzian at $\bar{x}\in S\cap\bd S$. Then, the Clarke normal cone of $S$ at $\bar{x}$ can be described as
$$N(S,\bar{x})=\R_+\partial\Delta_S(\bar{x}),$$
where $\R_+:=[0,+\infty[$.
\end{corollary}
\begin{theorem}\label{tcz2}\cite{10}
Let $g:X\rightarrow \R$ be a Lipschitz continuous function near $\bar{x}$ of the real normed space $X$ and let $S:=\{x\in X: g(x)\leq 0\}.$ Assume that $\bar{x}\in\bd S$ and $0\notin \partial g(\bar{x})$.
Then $S$ is epi-Lipschitzian at $\bar{x}$.
\end{theorem}
\begin{lemma}\label{epi}\cite{10}
Let $S$ be a subset of a normed space $X$ which is epi-Lipschitzian at $\bar{x}\in S\cap\bd S$. Then, the following equality holds
$$T(\bd S,\bar{x})=T(S,\bar{x})\cap-T(S,\bar{x}),$$
which yields in particular that $T(\bd S,\bar{x})$ is a closed vector subspace of $X$.
\end{lemma}
Using the Ioffe-approximate subdifferential $\partial_A$ (see \cite{14} and \cite{15}), the author \cite{18} established the following result.
\begin{proposition}[Corollary 4.1 \cite{18}] Let $X$ be a Banach space and  $g:X\rightarrow \R\cup\{+\infty\}$ be a lower semicontinuous  function  around $\bar{x}$, with $g(\bar{x}) = 0$. Let $S:=\{x\in \R^n: g(x)\leq 0\}.$ Assume that $\bar{x}\in\bd S$ and $0\notin \partial_A g(\bar{x})$.
Then $S$ is CEL at $\bar{x}$ provided that $\epi g$ is CEL at $(\bar{x}, 0)$.
\end{proposition}
It is well known that if $\bar{x}$ is an interior point of $S$ then $T(S, \bar{x}) = X$. Unfortunately the  reverse implication is not true in general. To see this, take $X$ a separable Hilbert space with orthonormal basis $(b_n)_n$ and consider (see. \cite{4}) the set $S = clco\{ {{b_n}\over{2^n}}, -{{b_n}\over{2^n}}\}$ with $\bar{x}=0$. Then $\mathrm{int} S = \emptyset$ and $T(S, \bar{x}) = X$. It is shown in \cite{16} that the reverse implication holds for the class of compactly epi-Lipschitz  (CEL) sets.  The main argument used in \cite{16} is based on the following result established in \cite{17}.
\begin{proposition}[Corollary 2.7 in \cite{17}]\label{prop1}  Let $S\subset X$ be a closed CEL set at $\bar{x}$ in the boundary of $S$. Then there exists $x^*\in X^*$, with $x^*\neq 0$, such that $x^*\in N(S, \bar{x})$. 
\end{proposition}
Let us prove the following theorem about epi-Lipschitz property of image and inverse image of epi-Lipschitz sets in Banach spaces
\begin{proposition}[Image and inverse image of epi-Lipschitz sets]\label{st} Let $X$ and $Y$ be Banach spaces, $S \subset Y$ be a closed set and $\varphi: X\rightarrow Y$ be a strictly Fr\'echet differentiable mapping at $\bar{x}$ with a surjective derivative $D\varphi(\bar{x})$. Then the set $C:=\varphi^{-1}(S)$ is epi-Lipschitzian at $\bar{x}\in \varphi^{-1}(S)$ IFF 
$S$ is epi-Lipschitz at $\varphi(\bar{x})$. 
\end{proposition}
\begin{proof} Suppose that $C$ is epi-Lipschitzian at $\bar{x}$. Since $\varphi$ is  strictly Fr\'echet differentiable  at $\bar{x}$ with a surjective derivative $D\varphi(\bar{x})$, Theorem 1.3 in \cite{11} ensures the existence of $\alpha >0$ and $r>0$ such that 
\begin{eqnarray}\label{Open}
B(\varphi(x), t) \subset \varphi(B(x, \alpha t)) \, \forall\, x\in B(\bar{x}, r), \, \forall\, t\in [0, r].
\end{eqnarray}
On the other hand since the set $C:=\varphi^{-1}(S)$ is epi-Lipschitzian at $\bar{x}$  there exist $\bar{v} \in X$, $\bar{v}\neq 0$,  and $s >0$ such that 
\begin{eqnarray}\label{lip}
C\cap B(\bar{x}, s) + tB(\bar{v}, s) \subset C \, \forall\, t\in [0, s]
\end{eqnarray} 
Moreover $\varphi$ strictly Fr\'echet differentiable  at $\bar{x}$ implies that for all $\varepsilon >0$ there exists $\delta >0$ such that 
\begin{eqnarray}\label{strictdiff}
\Vert \varphi(x) - \varphi(y) - D\varphi(\bar{x})(x-y)\Vert \leq \varepsilon\Vert x-y\Vert \, \forall\, x, y\in B(\bar{x},2 \delta).
\end{eqnarray}
We may assume that $\alpha \geq 1$,  $r=s$ and  $\alpha(\varepsilon \Vert \bar{v}\Vert + 2\delta) \leq {r\over 2}$. Pick $ \delta_1>0$, with $\alpha\delta_1\leq {\delta\over 2}$, and let $y\in B(\varphi(\bar{x}), \delta_1)\cap S$ and $b\in B(0, \delta) \subset Y$. Relation (\ref{Open}) ensures the existence of $x\in B(\bar{x}, \alpha \delta_1)$ such that $y = \varphi(x)$, so $x\in C\cap B(\bar{x}, \alpha \delta_1)$.\\ Now using (\ref{strictdiff}), we obtain 
$$\forall\,  t\in ]0, {\delta\over{\Vert \bar{v}\Vert}}], \quad \Vert \varphi(x+t\bar{v}) - \varphi(x) - tD\varphi(\bar{x})\bar{v}\Vert \leq \varepsilon t\Vert \bar{v}\Vert$$
and hence $$\varphi(x) +t(D\varphi(\bar{x})\bar{v} + b)\in B(\varphi(x+t\bar{v}), t(\varepsilon \Vert \bar{v}\Vert)+\delta).$$
By (\ref{Open}), there exists $b'\in B(0, \alpha(\varepsilon \Vert \bar{v}\Vert)+\delta))\subset X$ such that $$\varphi(x) +t(D\varphi(\bar{x})\bar{v} + b) = \varphi(x+t(\bar{v} + b')).$$ Invoking relation (\ref{lip}), we get $x+t(\bar{v} + b')\in  C$, equivalently, $\varphi(x+t(\bar{v} + b')) \in S$ and hence 
$$\varphi(x) +t(D\varphi(\bar{x})\bar{v} + b) \in S.$$
In summary, we have 
$$\forall\,  t\in ]0, {\delta\over{\Vert \bar{v}\Vert}}], \quad  B(\varphi(\bar{x}), \delta_1)\cap S + tB(D\varphi(\bar{x})\bar{v} , \delta) \subset S.$$
By taking $\gamma = \min\{{\delta\over{\Vert \bar{v}\Vert}}, \delta_1\}$ we deduce that $S$ is epi-Lipschitzian at $\varphi(\bar{x})$.
\\
Conversely, suppose that $S$ is epi-Lipschitzian at $\varphi(\bar{x})$. As $\varphi$ is strictly Fr\'echet differentiable at $\bar{x}$ with  surjective derivative $D\varphi(\bar{x})$, there exists $\alpha >0$ such that 
for all $\varepsilon > 0$ there exists $\delta > 0$ satisfying 
\begin{eqnarray}\label{Cont}
\varphi(B(\bar{x}, \delta)) \subset B(\varphi(\bar{x}), \varepsilon),
\end{eqnarray}
\begin{eqnarray}\label{strict}
\Vert \varphi(x) - \varphi(x') - D\varphi(\bar{x})(x-x')\Vert \leq \varepsilon \Vert x-x'\Vert \, \forall\, x, x'\in B(\bar{x}, \delta),
\end{eqnarray}
and 
\begin{eqnarray}\label{surject}
B(0, 1) \subset D\varphi(\bar{x})(B(0, \alpha)).
\end{eqnarray}
By assumption,  there exists $\gamma > 0$ and $w \in Y$ such that 
\begin{eqnarray}\label{epiL}
\forall\, t\in ]0, \gamma], \quad  B(\varphi(\bar{x}), \gamma)\cap S + tB(w , \gamma) \subset S.
\end{eqnarray}
Using the surjectivity of $D\varphi(\bar{x})$ there exists $v\in X$ such that $w = D\varphi(\bar{x})(v)$.  So, let $x\in B(\bar{x},\delta)\cap C$, $t\in ]0, \delta[$ and $u\in B(v, \delta)$. Relations (\ref{strict}), (\ref{Cont}) and (\ref{surject})  ensure the following inclusions
\begin{align*}
\varphi (x+tu) &\in \varphi(x) + tD\varphi(\bar{x})(u) + B(0, \varepsilon t\Vert u\Vert )\quad (\hbox{by }  (\ref{strict}))\\
 &  \subset B(\varphi(\bar{x}), \varepsilon)\cap S +  tD\varphi(\bar{x})(u) + B(0, \varepsilon t\Vert u\Vert )\quad (\hbox{by }  (\ref{Cont})) \\
 & \subset B(\varphi(\bar{x}), \varepsilon)\cap S + t D\varphi(\bar{x})(B(u, \alpha\varepsilon \Vert u\Vert)) \quad (\hbox{by }  (\ref{surject}))\\
 & \subset B(\varphi(\bar{x}), \varepsilon)\cap S + t (B(D\varphi(\bar{x})(u),  \alpha\varepsilon \Vert u\Vert\cdot \Vert D\varphi(\bar{x})\Vert)\\
 &  \subset S \quad (\hbox{because $\varepsilon$ is arbitrary and  (\ref{epiL})}).
\end{align*}
This shows that $C$ is epi-Lipschitzian at $\bar{x}$. The proof is completed. 
\end{proof} 
\begin{remark} This proposition tells us that we may replace the function $\varphi$ in \eqref{epiphi} by any strictly Fr\'echet differentiable mapping  $\varphi: E\times \mathbb{R}\rightarrow X$ at $\bar{y}:=(\bar{u}, f(\bar{u})) \in \varphi^{-1}(\bar{x})$ with a surjective derivative $D\varphi(\bar{y})$. Where $S$ is as in the first part of \eqref{epiphi}. 
\end{remark}
\begin{remark} Note that this proposition is not true even if $\varphi$ is a linear non surjective operator. Indeed, let (\cite{5}) $X = Y = l^2$ be the Hilbert space of square summable
sequences, with $(e_k)$ its canonical orthonormal base and let the operator $\varphi : l^2 \to l^2$ be defined by
$$\varphi(\sum x_i e_i) = \sum 2^{1-i}x_ie_i.$$
It is easily seen that for all $i$, $e_i\in \hbox{Im}\varphi$ and for all $k\geq 1$, $\displaystyle \sum_{i=1}^k  2^{1-i}e_i \in \hbox{Im}\varphi$ but $\displaystyle \sum 2^{1-i}e_i \notin \hbox{Im}\varphi$ (this shows that  Im$(\varphi)$ is not closed).
Then $\varphi$ is not surjective. Further Im$(\varphi)$ is a proper dense subspace of $l^2$. So that the space $X$ is epi-Lipschitzian at $0$, but not $\varphi(X)$. 
\end{remark}
\section{Strict Hadamard  differentiablity of epi-Lipschitz sets}
This section is devoted to the local representation of epi-Lipschitz set as the epigraph of strictly Hadamard  differentiable function. As a motivation of the following result we start with the following observation. 
Let  $S = \hbox{epi}\,f$ be the epigraph of the function $f$  defined by
$$
f(x)= \left\{\begin{array}{ll}
x^2\sin{1\over x} & \hbox{  if  }  x\neq 0\\
0& \hbox{  if  }  x=0
\end{array}
\right.
$$
and let $\bar{x}=(0,0)$. Then  $f$ is Lipschitz near $0$ and $T(S, \bar{x})  = \hbox{epi}\vert \cdot\vert.$ This shows that $T(S, \bar{x})$ does not contain any hyperplane.\\

\noindent In the sequel,  we will establish that if $f$ is Lipschitz near $\bar{x}$ then $T(\epi\,f,(\bar{x},f(\bar{x})))$ contains a closed  hyperplane if and only if $f$ is  strictly Hadamard  differentiable at $\bar{x}$. 

\begin{definition} We say that an epi-Lipschitz set $S$ at $\bar{x}\in S$ is strictly Hadamard  differentiable at $\bar{x}$ if its local representation $f$ is strictly Hadamard differentiable at $\bar{x}$. 
\end{definition}
\begin{theorem}\label{th0}
Let $X$ be a  real  Banach space and $S \subset X$ be a closed set containing $\bar{x}$. Suppose that $S$ is epi-Lipschitz at $\bar{x}$. Then $S$ is strictly Hadamard differentiable at $\bar{x}$ if and only if $T(S,\bar{x})$ contains a closed hyperplane.
\end{theorem}
\begin{proof} Since $S$ is epi-Lipschitz at $\bar{x}$, there exist 
\begin{itemize}
\item[$\bullet$] a closed hyperplane $E$ and $\bar{v}\in X\setminus\{0\}$ such that $X=E\oplus\R\bar{v}$,
\item[$\bullet$] a Lipschitz continuous mapping $f:E\rightarrow\R$ near $\bar{u}=\pi_E\bar{x}$ and
\item[$\bullet$]  a strictly Fr\'echet differentiable mapping $\varphi: E\times\R\rightarrow X$   at $\bar{y}=(\bar{u},f(\bar{u}))$ with a surjective derivative $D\varphi(\bar{y})$ 
\end{itemize}
such that 
\begin{equation}\label{st1}
S\cap\Omega=\Omega\cap\varphi(\mathrm{epi}\,f)
\end{equation}
\noindent Since $f$ is Lipschitz continuous around $\bar{u}$, it is Clarke subdifferentiable at $\bar{u}$. Moreover,  by \eqref{st1}  and the fact that $ D\varphi(\bar{u},f(\bar{u}))$ is surjective,  we have (see for example Theorem 1.17 in \cite{21})
\begin{eqnarray}\label{equaN}
N(\mathrm{epi}\,f,(\bar{u},f(\bar{u}))) = [D\varphi(\bar{y})]^*N(S,\bar{x})
\end{eqnarray}
where $[D\varphi(\bar{y})]^*$ denotes the adjoint linear mapping of $D\varphi(\bar{y})$.\\
\noindent This implies that 
\begin{equation}\label{st2}
\partial f(\bar{u})\times\{-1\}\subset [D\varphi(\bar{y})]^*N(S,\bar{x}).
\end{equation}
Now suppose that $T(S,\bar{x})$ contains a closed hyperplane. Let us show that $f$ is strictly Hadamard  differentiable at $\bar{u}$.\\
Taking the polar of $T(S,\bar{x})$ and using \eqref{st2} we have 
$$\partial f(\bar{u})\times\{-1\}\subset \R_+[D\varphi(\bar{y})]^*(a^*)$$ for some $a^*\neq0$.  
Let $[D\varphi(\bar{y})]^*(a^*)=(a_1^*,a_2^*)$ with $a_1^*\in E^*$ and $a_2^*\in\R$, we have
$$\partial f(\bar{u})\times\{-1\}\subset\R_+(a_1^*,a_2^*).$$
Now for $u^*\in\partial f(\bar{u})$, there exists $\lambda>0$ such that $(u^*,-1)=\lambda(a_1^*,a_2^*)$.\\
This implies that $a_2^*=-{1\over\lambda}$. Moreover $u^*$ is constant with $u^*=-{a_1^*\over a_2^*}$. That is 
$$\partial f(\bar{u})=\left\{-{a_1^*\over a_2^*}\right\}.$$
Since $f$ is Lipschitz continuous in an open neighbourhood of $\bar{u}$ and  $\partial f(\bar{u})$ is a singleton then $f$ is strictly Hadamard  differentiable at $\bar{u}$.\\
\noindent Conversely suppose that $f$ is strictly Hadamard differentiable at $\bar{u}$. Let us show that $T(S,\bar{x})$ contains a closed hyperplane.\\
Relation (\ref{equaN}) ensures that  $ N(S,\bar{x}) = \R_+{[D\varphi(\bar{y})]^*}^{-1}(\nabla f(\bar{u}),-1)$.  Since $D\varphi(\bar{y})$ is surjective, the set ${[D\varphi(\bar{y})]^*}^{-1}(\nabla f(\bar{u}),-1)$ is a singleton. Put $a^* = {[D\varphi(\bar{y})]^*}^{-1}(\nabla f(\bar{u}),-1)$. Then $a^*\neq 0$ and 
$$T(S,\bar{x})=\{h\in X:\langle a^*,h\rangle\leq0\}$$
which clearly contains a hyperplane.
\end{proof}

In what follows,  we shall prove that the strict Hadamard  differentiability of $S$ at $\bar{x}$  is equivalent to saying that $T(\bd S,\bar{x}) $ is a closed hyperplane. We begin with the next lemma.
\begin{lemma}\label{lem3}
Let $S$ be a closed subset of a real Banach (resp. real normed vector) space $X$. Assume that $S$ is CEL (resp. epi-Lipschitzian) at $\bar{x}\in \bd S$. If $T(S,\bar{x})$ contains a closed hyperplane $H$, then $T(S,\bar{x})\cap (-T(S,\bar{x})) = H$ (resp. $T(\bd S, \bar{x}) = H$). 
\end{lemma}
\begin{proof}
Let $X$ be a real Banach space. Suppose that $S$ is CEL at $\bar{x}$ and that $T(S,\bar{x})$ contains a closed hyperplane $H$. Assume that there exists $h\in T(S,\bar{x})\cap (-T(S,\bar{x})) $ such that $h\notin H$. By the Banach strict separation there exist $x^*\neq0$ and $\alpha\neq0$ such that 
$$\langle x^*,h\rangle>\alpha>\langle x^*,v\rangle,\;\forall\;v\in H.$$
This implies that $\langle x^*,v\rangle=0$ for all $v\in H$ and $\alpha>0$. Consequently we have
$x^*\in H^\perp$. Since $H\subset T(S,\bar{x})$ we have $N(S,\bar{x})\subset H^\perp$. This implies that $$N(S,\bar{x})\subset \R x^*.$$
Moreover since $\bar{x}\in\bd S$ and $S$ is CEL at $\bar{x}$, Proposition \ref{prop1} ensures that  $N(S,\bar{x})\neq\{0\}$. Therefore 
$$x^*\in N(S,\bar{x})\quad\mbox{or}\quad-x^*\in N(S,\bar{x}).$$
Because $h\in T(S,\bar{x})\cap (-T(S,\bar{x})) $, any one of the two last conditions imply that $\langle x^*,h\rangle=0$ which contradicts the fact that $\alpha>0$. 
Therefore $$T(S,\bar{x})\cap (-T(S,\bar{x})) =H.$$
Now let $X$ be a real normed space. Suppose that $S$ is epi-Lipschitzian and that $T(S,\bar{x})$ contains a closed hyperplane $H$. 
Assume that there exists $h\in T(S,\bar{x})\cap (-T(S,\bar{x})) $ such that $h\notin H$. 
Following the above arguments with the Banach strict separation theorem there exists some $x^*\neq0$ such that $$N(S,\bar{x})\subset \R x^*.$$ Moreover Theorem \ref{tcz1} and Corollary \ref{cor1} imply that $N(S,\bar{x})\neq\{0\}.$\\
Similar arguments as above give that $T(S,\bar{x})\cap (-T(S,\bar{x}))=H$. Therefore Lemma \ref{epi} imply that $T(\bd S, \bar{x}) = H$.
\end{proof}
So, we have the following result which follows immediately from the above lemma and Theorem \ref{th0}.
\begin{theorem}
Let $S$ be a closed subset of the real Banach space $X$ which  is epi-Lipschitzian at $\bar{x}\in \bd S$. Then $S$ is strictly Hadamard differentiable at $\bar{x}$ if and only if $T(\bd S,\bar{x}) $ is a closed hyperplane. 
\end{theorem}
We complete this section by the two following corollaries of Theorem \ref{th0}.
\begin{corollary}
Let $S$ be a closed subset in a real Banach space $X$. Assume that $S$ is epi-Lipschitzian at $\bar{x}\in \bd S$. Then $S$ is strictly  Hadamard differentiable at $\bar{x}$ if and only if $\partial \Delta_S(\bar{x})$ is contained in a segment $]0,a^*]$ for some $a^*\in X^*$.
\end{corollary}
\begin{proof}
Suppose that $S$ is strictly Hadamard differentiable at $\bar{x}$. Then by Theorem \ref{th0}, there exists a closed hyperplane $H$ such that $H\subset T(S,\bar{x})$. This implies that
$$N(S,\bar{x})\subset \R x^*\quad\mbox{for some}\quad x^*\in X^*\setminus\{0\}.$$
This and Corollary \ref{cor1} imply that
$$\R_+\partial\Delta_S(\bar{x})\subset \R x^*.$$
Moreover by Theorem \ref{tcz1},  we have $0\notin\partial\Delta_S(\bar{x})$.  This and the fact $\partial\Delta_S(\bar{x})$ is bounded implies that $$\partial \Delta_S(\bar{x})\subset ]0,a^*]\quad\mbox{for some}\;a^*\neq0.$$
Conversely assume that $\partial \Delta_S(\bar{x})\subset ]0,a^*]$ for some $a^*\neq0$. Then we have 
$$\R_+\partial \Delta_S(\bar{x})\subset \R a^*.$$ This implies that 
$$\langle a^*\rangle^\perp\subset T(S,\bar{x}).$$
That is $T(S,\bar{x})$ contains a closed hyperplane. Therefore by Theorem \ref{th0}, $S$ is strictly Hadamard differentiable at $\bar{x}$.
\end{proof}
\begin{corollary}
Let $g:X\rightarrow \R$ be a Lipschitz continuous function near $\bar{x}$ in the real Banach space $X$ and let $S:=\{x\in X: g(x)\leq 0\}.$ Assume that $\bar{x}\in\bd S$ and $0\notin \partial g(\bar{x})$.
Then $S$ is strictly Hadamard differentiable at $\bar{x}$ if and only if there exists some $a^*\neq 0$ such that $$\partial g(\bar{x})\subset ]0,a^*].$$
\end{corollary}
\begin{proof}
Suppose that $S$ is strictly Hadamard differentiable. By Theorem \ref{tcz2}, we have $S$ is epi-Lipschitzian at $\bar{x}$. Therefore by Theorem \ref{th0}, there exists a closed hyperplane $H$ such that $H\subset T(S,\bar{x})$. 
This implies that
$$N(S,\bar{x})\subset \R x^*\quad\mbox{for some}\quad x^*\in X^*\setminus\{0\}.$$
That is
$$\R_+\partial g(\bar{x})\subset \R x^*.$$
Since $\partial g(\bar{x})$ is bounded and not containing $0$ we have 
$$\partial g(\bar{x})\subset ]0,a^*]\quad\mbox{for some}\;a^*\in X^*\setminus\{0\}.$$
Conversely if $\partial g(\bar{x})\subset ]0,a^*]\quad\mbox{for some}\;a^*\neq0.$ Since $0\notin \partial g(\bar{x})$, Theorem \ref{tcz2} implies that $S$ is epi-Lipschitzian at $\bar{x}$. Moreover 
$$N(S,\bar{x}) \subset \R a^*.$$
This implies that $T(S,\bar{x})$ contains a closed hyperplane. Therefore Theorem \ref{th0} implies that $S$ is strictly Hadamard differentiable at $\bar{x}$.
\end{proof}
\section{Strict Hadamard  differentiability of non epi-Lipschitz sets}
This section is concerned with the strict Hadamard differentiability of non epi-Lipschitz sets in finite dimensions. 
We start with the next theorem.
\begin{theorem}\label{th1}
Let $X$ be a finite dimensional real vector space and $S\subset X$ be a closed subset. Assume that $S$ is non epi-Lipschitzian at $\bar{x}\in S$ and that $T(S,\bar{x})$ contains a closed hyperplane $H$. Then $T(S,\bar{x})$ coincides with $H$.
\end{theorem}
\begin{proof}
Assume that $\mathrm{Int}\,T(S,\bar{x})=\emptyset$ and that there exists a closed hyperplane $H$ such that $H\subset T(S,\bar{x})$.
\noindent Let $h\in T(S,\bar{x})$ such that $h\notin H$. By the Banach separation theorem there exists $x^*\in X$ and $\alpha\in\R$ such that
\begin{equation}\label{p1}
\langle x^*,h\rangle\, >\alpha\geq\langle x^*,u\rangle,\;\forall\;u\in H.
\end{equation}
This implies that
\begin{equation}\label{p2}
\langle x^*,u\rangle=0,\;\forall\;u\in H\quad\mbox{and}\quad\alpha\geq0.
\end{equation}
From \eqref{p1} and \eqref{p2} we have respectively:
$$x^*\notin N(S,\bar{x})\quad\mbox{and }\quad x^*\in H^\perp.$$
Let $a^*\neq0$ be a vector in $X$ such that $H^\perp=\R a^*$. Assume without loss of generality that $x^*\in \R_+a^*$.
Since $N(S,\bar{x})\subset H^\perp$ and  $x^*\notin N(S,\bar{x})$, we deduce that $N(S,\bar{x})=\R_{-}a^*$. Consequently
$$T(S,\bar{x})=\left\{h\in X:\langle a^*,h\rangle\geq0\right\}.$$
This is a contradiction because $\mathrm{Int}\,T(S,\bar{x})=\emptyset$. Therefore $T(S,\bar{x})=H$.
\end{proof}
The following theorem gives a geometric characterization of the strict Hadamard differentiability of mappings.
\begin{theorem}\label{strictmap} Let $X$ be a finite dimensional real vector space and  $Y$ a real Banach space. Let $g : X\mapsto Y$ be a mapping which is locally Lipschitzian around $\bar{x}$  and $\varphi: X\times Y\rightarrow Z$ be a strictly Hadamard differentiable mapping at $(\bar{x}, \bar{y})$ with a bijective derivative $D\varphi(\bar{x}, \bar{y})$ where $Z$ is a Banach space and  $\bar{y} = g(\bar{x})$. Let $S \subset Z$ be a closed set such that  $S \cap \Omega = \varphi(\gh g) \cap \Omega$ with $\Omega \subset Z$ an open set containing $\bar{z} := \varphi(\bar{x}, \bar{y})$. Then the following assertions are equivalent 
\begin{itemize}
\item[$1)$] $g$ is strictly Hadamard differentiable at $\bar{x}$,
\item[$2)$] $T(S, \bar{z})$ is isomorphic to $X$. 
\end{itemize}
\end{theorem} 
\begin{proof}  $1) \Rightarrow 2)$ :  Suppose that  $g$ is strictly Hadamard differentiable at $\bar{x}$. Then  the following equality holds
\begin{eqnarray}\label{tang}
T(\gh(g), (\bar{x}, g(\bar{x}))) = \gh Dg(\bar{x}).
\end{eqnarray}
Indeed, let $(h, k) \in T(\gh(g), (\bar{x}, g(\bar{x})))$. Then for all $t_n\to 0^+$ there exists $h_n \to h$ and $k_n\to k$ such that 
$$g(\bar{x}+t_nh_n) = g(\bar{x}) +t_nk_n, \hbox{  for $n$ sufficiently large}.$$
Using the differentiability and the local Lipschitzness of $g$ at $\bar{x}$, we get
$$Dg(\bar{x})h = \lim_{n\to +\infty}{{g(\bar{x}+t_nh) - g(\bar{x}) }\over{ t_n}} = k.$$
Conversely, let $h\in X$, $(x_n) \subset X$, with $\displaystyle \lim_{n\to +\infty} x_n = \bar{x}$ and $t_n\to 0^+$. Then, by the strict Hadamard  differentiability of $g$
$$k:= Dg(\bar{x})h = \lim_{n\to +\infty}{{g(x_n+t_nh) - g(x_n)}\over{t_n}}.$$
Put $k_n = {{g(x_n+t_nh) - g(x_n)}\over{t_n}}$. Then $(x_n, g(x_n))+t_n(h, k_n) \in  \gh g$ for all $n$. Thus, $(h, k) \in T(\gh g, (\bar{x}, g(\bar{x})))$. \\
Since $D\varphi(\bar{x}, \bar{y})$ is bijective, we have 
\begin{eqnarray}\label{iso}
T(S, \varphi(\bar{x}, \bar{y})) = D\varphi(\bar{x}, \bar{y}) (T(\gh(g), (\bar{x}, g(\bar{x})))).
\end{eqnarray}
Now we define the mapping $\psi : X \mapsto T(S, \varphi(\bar{x}, \bar{y}))$ by 
$$\psi(h) = D\varphi(\bar{x}, \bar{y})(h, Dg(\bar{x})h).$$
Using relation (\ref{tang}), it is easy to see that $\psi$ is an isomorphism from $X$ into $T(S, \varphi(\bar{x}, \bar{y}))$. \\
\\
$2) \Rightarrow 1)$:   Let $\xi :  X \mapsto T(S, \varphi(\bar{x}, \bar{y}))$ be an isomorphism. Using relation (\ref{iso}), the mapping $\psi := D\varphi(\bar{x}, \bar{y})^{-1}\circ \xi$ is an isomorphism from $X$ to 
$T(\gh(g), (\bar{x}, g(\bar{x})))$. Let $\psi_1 : X\mapsto X$ and $\psi_2 : X\mapsto Y$ be two linear mappings such that $\psi(x) = (\psi_1(x), \psi_2(x))$ for all $x\in X$. We have for all  $h\in X$, $(\psi_1(h),\psi_2(h)) \in T(\gh(g), (\bar{x}, g(\bar{x})))$. Let $h \in X$, $x_n\to \bar{x}$  and $t_n\to 0^+$. Then  there exist $h_n \to\psi_1(h)$ and  $ k_n \to\psi_2(h)$ such that $(x_n, g(x_n)) + t_n (h_n, k_n) \in \gh g$, for all integer $n$ sufficiently large. So 
\begin{eqnarray}\label{equa}
\lim_{n\to +\infty} {{g(x_n+t_n\psi_1(h)) - g(x_n)}\over{t_n}} =\psi_2(h).
\end{eqnarray}
 We claim that the kernel, ker$\psi_1$, of $\psi_1$ is reduced to the singleton $\{0\}$. Because $\psi$ is an isomorphism, it is easy to see that $\hbox{ker}\psi_1\cap \hbox{ker}\psi_2 = \{0\}$.  Let $h\in \hbox{ker}\psi_1$. Then relation (\ref{equa}) ensures that $\psi_2(h) = 0$ and hence $h=0$. This means that $\psi_1$ is an isomorphism. 
 Put $\gamma =\psi_2\circ\psi_1^{-1}$. We claim that $\gamma $ is the strict Hadamard  derivative of $g$ at $\bar{x}$, that is,
 $$\lim_{{x\to \bar{x}}\atop{x'\to \bar{x}}} {{g(x) - g(x') - \gamma (x-x')}\over{\Vert x-x'\Vert}} = 0.$$
 So suppose the contrary, that is, there exists $\varepsilon >0$ such that 
 $$\limsup_{{x\to \bar{x}}\atop{x'\to \bar{x}}}\Vert {{g(x) - g(x') - \gamma (x-x')}\over{\Vert x-x'\Vert}} \Vert\geq\varepsilon .$$
 Let $u_n\to \bar{x}$ and $u'_n\to \bar{x}$ be such that 
$$\lim_{n\to +\infty}\Vert {{g(u_n) - g(u'_n) - \gamma (u_n-u'_n)}\over{\Vert u_n-u'_n\Vert}} \Vert \geq \varepsilon.$$
Take $t_n:= \Vert u_n-u'_n\Vert$ and $w_n = {{u_n-u'_n}\over {\Vert u_n-u'_n\Vert}}.$ Since $X$ is of finite dimension, extracting subsequence if necessary, we may assume that $w_n\to w$, with $\Vert w\Vert = 1.$ So
$$\lim_{n\to +\infty}\Vert {{g(u'_n+t_nw_n) - g(u'_n) - t_n\gamma (w_n)}\over{t_n}} \Vert \geq \varepsilon$$
or equivalently (because $g$ is locally Lipschitzian around $\bar{x}$) 
$$\lim_{n\to +\infty}\Vert {{g(u'_n+t_nw) - g(u'_n) - t_n\gamma (w_n)}\over{t_n}} \Vert \geq  \varepsilon.$$
Since $\psi_1$ is an isomorphism, there exists a unique $h\in X$ such that $w =\psi_1(h)$. Using relation (\ref{equa}), we obtain a contradiction with the last inequality. The proof is then completed.  
\end{proof}
Let us stress in the following remark an equivalence result of the strict  Hadamard differentiability for sets in  finite dimensional spaces regardless of wether the set is epi-Lipschitzian or not.
\begin{remark} Let $X$ be a finite dimensional space and $g : X\mapsto \R$ be a locally Lipschitz function around $\bar{x}$. Then $g$ is strictly Hadamard differentiable at $\bar{x}$ IFF $T(\gh g, (\bar{x}, g(\bar{x})))$ is isomorphic to $X$ and  here $X\times \{ 0\}$ is considered as an  hyperplane of $X \times \R$.\\
Therefore, let $S\subset X$ be closed subset having $g$ as its local representation. Then $S$ is strictly Hadamard differentiable at $\bar{x}$ if and only if $T(S,\bar{x})$ is isomorphic to $X$.
\end{remark}

Now, we may establish a geometrical characterization in the spirit of \cite{2} involving the contingent or Bouligand  cone. We recall that the contingent cone to a closed set $S$ at $\bar{x} \in S$  is defined as the following limit superior of the set-differential quotient
 $$K(S,\bar{x}):= \limsup_{t\to 0^+}{{S-\bar{x}}\over t}.$$
When $X$ is a finite dimensional space the contingent cone and the Clarke tangentcone are connected in the following way(see \cite{9} and \cite{23})
\begin{eqnarray}\label{KT}
 T(S, \bar{x}) =  \liminf_{\overset{S}{x\to \bar{x}}} K(S, x).
\end{eqnarray}
We say that a set-valued mapping $F : X\rightrightarrows Y$ is continuous at $\bar{x}$ if 
$$\lim_{x\to \bar{x}} F(x) = F(\bar{x})$$
where 
$$\lim_{x\to \bar{x}} F(x)  : = \limsup_{x\to \bar{x}} F(x)  = \liminf_{x\to \bar{x}} F(x).$$
Here 
$$\limsup_{x\to \bar{x}} F(x) := \{ h\in X :\, \liminf_{x\to \bar{x}} d_{F(x)}(h) = 0\}$$
and
$$\liminf_{x\to \bar{x}} F(x) := \{ h\in X :\, \lim_{x\to \bar{x}} d_{F(x)}(h) = 0\}.$$

\begin{corollary} Let $X$, $Y$ and $Z$ be  finite dimensional spaces, $g : X\mapsto Y$ be a mapping which is locally Lipschitzian around $\bar{x}$  and $\varphi: X\times Y\rightarrow Z$ be 
a strictly Hadamard differentiable mapping at $(\bar{x}, \bar{y})$ with a bijective derivative $D\varphi(\bar{x}, \bar{y})$   
 and  put $\bar{y} = g(\bar{x})$. Let $S \subset Z$ be a closed set such that  $S \cap \Omega = \varphi(\gh g) \cap \Omega$, where $\Omega \subset Z$ is an open set containing $\bar{z} := \varphi(\bar{x}, \bar{y})$. Then the following assertions are equivalent 
\begin{itemize}
\item[$1)$]  $g$ is strictly Hadamard differentiable at $\bar{x}$,
\item[$2')$] The set-valued mapping $z\rightrightarrows K(S, z)$ is continuous at $\bar{z}$ with respect to $S$ and $K(S, \bar{z})$ is isomorphic to $X$. 
\end{itemize}
\end{corollary} 
\begin{proof} Using relation (\ref{KT}), it is easy to see that $2')$ implies the assertion $2)$ of Theorem \ref{strictmap}.   So it is enough to prove the implication $1) \Longrightarrow 2')$. Without loss of generality, we may assume that $\varphi = id_{X\times Y}$ and $Z = X\times Y$, where $id_{X\times Y}$ is the identity mapping of $X\times Y$. It is easy to see that 
\begin{eqnarray}\label{Kg}
K(\gh g, (\bar{x}, g(\bar{x}))) = \gh Dg(\bar{x}).
\end{eqnarray}
We will establish the following inclusions 
$$ \limsup_{x\to \bar{x}} K(\gh g, (x, g(x))) \subset K(\gh g, (\bar{x}, g(\bar{x}))) \subset \liminf_{x\to \bar{x}} K(\gh g, (x, g(x))).$$
So let $\displaystyle (u, v) \in \limsup_{x\to \bar{x}} K(\gh g, (x, g(x)))$. Then there are sequences $x_n\to \bar{x}$ and $(u_n, v_n)\to (u, v)$ such that for $n$ large enough $(u_n, v_n) \in K(\gh g, (x_n, g(x_n)))$. Thus for each integer $n$ there exist sequences $t^n_k \to 0^+$ and $(u^n_k, v^n_k) \to (u_n, v_n)$ such that $(x_n, g(x_n))+t^n_k(u^n_k, v^n_k) \in \gh g$ for all $k$ sufficiently large. Extracting a diagonal subsequence, we may assume that for some subsequence of integers $(k_n)_n$, we have $\displaystyle \lim_{n\to +\infty}t^n_{k_n} = 0$, $\displaystyle \lim_{n\to +\infty}(u^n_{k_n}, v^n_{k_n}) = (u, v)$ and 
$(x_n, g(x_n))+t^n_{k_n}(u^n_{k_n}, v^n_{k_n}) \in \gh g$ for all $n$ sufficiently large. Thus, since $g$ is strictly Hadamard  differentiable at $\bar{x}$, 
$$Dg(\bar{x}) u =  \lim_{n\to +\infty}{{g(x_n+t^n_{k_n}u^n_{k_n}) - g(x_n)}\over{t^n_{k_n}}} = v.$$
This last equality ensures that $(u, v) \in K(\gh g, (\bar{x}, g(\bar{x})))$. \\
To prove the second inclusion 
$$K(\gh g, (\bar{x}, g(\bar{x}))) \subset \liminf_{x\to \bar{x}} K(\gh g, (x, g(x))),$$ it suffices to use relations (\ref{Kg}), (\ref{KT}) and (\ref{tang}). 
\end{proof}
\vskip 0.3cm
\noindent
{\bf Acknowledgment.} The authors would like to thank the referee for useful comments
which led to an improvement of their original manuscript. Abderrahim Jourani is partially supported by the EIPHI Graduate School (contract ANR-17-EURE-0002).

\end{document}